\documentclass[11pt,a4paper]{article}
\usepackage{amsfonts,amsbsy,amssymb,amsmath,graphicx}
\usepackage{epsfig}
\usepackage{dsfont}
\usepackage{mathrsfs}
\usepackage{geometry}
\usepackage[T1]{fontenc}
\usepackage{subeqnarray}
\usepackage{enumerate}
\font\bfit = cmbxti10

\usepackage{color}
\usepackage{rotating}
\usepackage{epsfig}
\usepackage{ifpdf}
\usepackage{cite}
\usepackage{setspace}
\usepackage[colorlinks=true,linkcolor=blue,citecolor=blue]{hyperref}

\newenvironment{proof} {\noindent {\em \textbf{Proof}} } { \hfill \fbox{~} \\ }

\def\1{1\kern-.20em {\rm l}}
\def\abs  #1\par{\vskip2truemm   {\noindent} {\parindent=12truemm\narrower\baselineskip=3truemm \pf #1\par}}

\newtheorem{theorem}{Theorem}[section]

\newtheorem{corollary}[theorem]{Corollary}

\newtheorem{definition}[theorem]{Definition}

\newtheorem{lemma}[theorem]{Lemma}

\newtheorem{proposition}[theorem]{Proposition}

\numberwithin{equation}{section}

\newcommand{\R}{\mathbb{R}}
\newcommand{\N}{\mathbb{N}}

\newcommand{\C}{\mathbb{C}}

\def\1{1\kern-.20em {\rm l}}
\def\qed{\ifmmode\mbox{\hfill\sqb}\else{\ifhmode\unskip\fi%
\nobreak\hfil
\penalty50\hskip1em\null\nobreak\hfil\sqb
\parfillskip=0pt\finalhyphendemerits=0\endgraf}\fi}
\def\keywords#1{\par
	\vspace*{8pt}
	{{\leftskip18pt\rightskip\leftskip
	\noindent{\it Keywords}\/:\ #1\par}}\par}
\def\cqfd{\ifmmode\sqw\else{\ifhmode\unskip\fi\nobreak\hfil
\penalty50\hskip1em\null\nobreak\hfil\sqw
\parfillskip=0pt\finalhyphendemerits=0\endgraf}\fi}
\title{\bf On the $D_{\omega}$-classical orthogonal polynomials}
\author{\small  Khalfa DOUAK \\
\small Laboratoire Jacques-Louis Lions, Sorbonne Université,\\
4 place Jussieu 75252 Paris Cedex 05, France\\
\small  khalfa.douak@gmail.com\\
\\
{\bfit Dedicated to Pascal MARONI for his 90th birthday}
}
\date{\today}
\begin{document}
\maketitle
\begin{abstract}
\noindent
 We wish to investigate the $D_{\omega}$-classical orthogonal polynomials, where $D_{\omega}$ is a special case of the Hahn operator. For this purpose,
 we consider the problem of finding all sequences of orthogonal polynomials such that their $D_{\omega}$-derivatives are also orthogonal polynomials.
 To solve this problem we adopt a different approach to those employed in this topic. We first begin by determining the coefficients involved in their recurrence relations, and then providing an exhaustive list of all solutions. When  $\omega=0$, we rediscover the classical orthogonal polynomials of Hermite, Laguerre, Bessel and Jacobi. For $\omega=1$, we encounter the families of discrete classical orthogonal polynomials as particular cases.
\end{abstract}

\keywords{ Classical orthogonal polynomials; discrete orthogonal polynomials; recurrence relations; difference operator; difference equations.}

 {\bf AMS Classification.}  33C45; 42C05.

\section{Introduction and preliminary results}
The orthogonal polynomials are characterized by the fact that they satisfy a second-order recurrence relation. They said to be classical if their
derivatives also form a sequence of orthogonal polynomials \cite{Hahn1}. Hahn generalized the classical orthogonal polynomials by generalizing their
characteristic properties (see \cite{AlSa,Chih} for more details). For this, he considered the linear operator \cite{Hahn2}
 \begin{equation}
\left(H_{q,\omega}f\right)(x):=\frac{f(qx+\omega)-f(x)}{(q-1)x+\omega},\label{eq1.1}
\end{equation}
for all polynomial $f$, with $q$ and $\omega$ are two fixed complex numbers.\\
Hahn showed that there is no loss of generality in assuming $\omega$ to be zero so that in what follows $q$ may be thought of as $1$ or different to $1$.
For $q\ne1$ and $\omega=0$, we obtain the $q$-difference operator (also known as the Jackson’s $q$-operator) which we write 
$\left({\mathscr D}_{q}f\right)(x):=\left(H_{q,0}f\right)(x)$.
When $q=1$ with $\omega\ne0$ we get the discrete operator $\left(D_{\omega}f\right)(x):=\left(H_{1,\omega}f\right)(x)$, that is,
\begin{equation}
\left(D_{\omega}f\right)(x)=\frac{f(x+\omega)-f(x)}{\omega}.\label{eq1.2}
\end{equation}
 For $\omega=1$, we meet the finite (or forward) difference operator $\Delta f(x)= f(x+1)-f(x)$.\par
 \smallskip
\noindent The limiting case $\omega\to0$ (resp. $q\to1$) of $D_{\omega}$ (resp. ${\mathscr D}_{q}$) gives rise to the derivative operator $D=d/dx$,
 giving  $(Df)(x):=f'(x)$. Because it is always possible to take such  a limit, this point is not really important at this time.
It will be dealt with a little bit later  when necessary.\par
\medskip
\noindent Motivated by the several properties common to all of the classical orthogonal polynomials, Hahn \cite{Hahn2} posed and solved five
(equivalent) problems that  are related to the operator ${\mathscr D}_{q}$ and obtained that all possible solutions lead to the same orthogonal
 polynomial sequences (OPS) which are the so-called classical $q$-orthogonal polynomials. Later on, the study of such polynomials has known an
 increasing interest (see for instance the \cite{KoLeSw} and the references therein).\\
The first problem studied by Hahn is the following:\par
\smallskip
\noindent Find all OPS $\{P_n\}_{n\geqslant0}$ such that $\{{\mathscr D}_{q}P_n\}_{n\geqslant0}$ is also an OPS.\par
\smallskip
\noindent For more details about the solutions of these problems we refer the reader to \cite{Hahn1,AlSa,Chih,Isma}.\\
In \cite{DoMa1} Douak and Maroni considered the problem of finding  all OPS such that their $D$-derivatives are also OPS. Instead of basing the study 
of this problem on the various properties of  orthogonal polynomials, the authors have rather founded their exposures in a purely algebraic point of view, focusing primarily on the explicit calculation of recurrence coefficients. The identified polynomials are none other than the classical orthogonal polynomials of Hermite, Laguerre, Jacobi and Bessel with the usual restrictions on the parameters.\par
\smallskip
\noindent Referring back to the operator $D_\omega$, we will pose the analogous problem: \par
 \smallskip
\noindent $\mathbf{(P)}$ {\it Find all OPS $\{P_n\}_{n\geqslant0}$ such that $\{D_{\omega}P_n\}_{n\geqslant0}$ are also OPS}.\par
\medskip
\noindent Note that in this field a  general method of studying the classical orthogonal polynomials of a discrete variable as solutions of 
a second-order difference equation of hypergeometric type was considered by Nikiforov {\it et al.} \cite{NiSuUv}. This approach was also adopted
 by Lesky in \cite{Lesk1,Lesk2}. \par
\medskip
\noindent This work is mainly intended to constructing the $D_{\omega}$-classical orthogonal polynomials by proceeding as in \cite{DoMa1}.
Such approach is rather new and, of course, different  to those previously used in several studies dedicated to this topic (see for instance
\cite{Hahn1,NiSuUv,GaMaSa,AbMa} and the references therein). After determining the recurrence coefficients, we proceed to the identification of 
the resulting polynomials. Under some restrictions on the parameters, we establish that these polynomials can be reduced to one of the well-known
 families of discrete classical orthogonal polynomials.\\
The same method was also used within the $d$-orthogonality context ($d\geqslant1$) to provide many extensions of the classical orthogonal polynomials
(see, e.g. \cite{DoMa1,DoMa2,Douak,LoVa} and the references  therein).\par
\medskip
\noindent In an earlier survey, Abdelkarim and Maroni \cite{AbMa} investigated the problem $\mathbf{(P)}$ according to a functional approach.
The authors had established various equivalent properties characterizing the resulting polynomials.  Particularly, they showed  that those polynomials satisfy the so-called functional Rodrigues's formula \eqref{eq1.14}. Based on this last characterization, up to linear transformation of the variable, 
they found that there are four classes of  $D_\omega$-classical orthogonal sequences satisfying Rodrigues's formula, including the Charlier, Meixner, Krawchuk and Hahn polynomials as  special cases of them.\par
 \medskip
\noindent Let $\mathscr P$ be the vector space of polynomials of one variable with complex coefficients and let $\mathscr P'$ be its algebraic dual.
 We denote by $\bigl<.\,,\, .\bigr>$ the duality brackets between $\mathscr P'$ and  $\mathscr P$. Let us denote by $\{P_n\}_{n\geqslant0}$ a polynomials
 sequence (PS), $\deg P_n=n$, and $\{u_n\}_{n\geqslant0}$ its associated dual sequence (basis) defined by $\bigl< u_n , P_m \bigr> = \delta_{n m} ;\, 
 n , m\geqslant0,$  where $\delta_{n m}$ is the Kronecker's delta symbol. The first element $u_0$ of the dual sequence is said to be the {\it canonical} 
 form associated to the PS $\{P_n\}_{n\geqslant0}$. Throughout this article, we will always consider the sequence of {\it monic} polynomials, i.e. the leading coefficient of each polynomial $P_n$ is one ($P_n(x)=x^n+\cdots$).\par
\noindent Given a form $u\in \mathscr P'$. The sequence of complex numbers $(u)_n,\ n=0, 1, 2, \ldots,$ defined by $(u)_n := \big< u , x^n \big>$
denotes the moments of $u$ with respect to the sequence $\{x^n\}_{n \geqslant0}$.
 The form $u$ is called {\it regular} (or {\it quasi-definite}) if we can associate with it a PS $\{P_n\}_{n\geqslant0}$ such that
\begin{align*}
\left<u,P_nP_m\right>=k_n\delta_{n,m},\, n, m\geqslant0\ ; \ k_n\ne0,\ n\geqslant0.
\end{align*}
In this case $\{P_n\}_{n\geqslant0}$ is an orthogonal polynomials sequence (OPS) with respect to (w.r.t.) $u$. As an immediate consequence of the regularity of $u$, we have  $(u)_0 \ne0$  and $u=\lambda u_0$ with $\lambda\ne0$. Furthermore, the elements of the dual sequence $\{u_n\}_{n\geqslant0}$ are such that
\begin{align}
u_n=\left(\big<u_0,P^2_n\big>\right)^{-1}P_nu_0,\ \ n=0, 1, 2, \ldots.\label{eq1.3}
\end{align}
  So, in all what follows, we  consider the orthogonality of any PS w.r.t. its canonical form $u_0$.\\
 First, let us  introduce the two operators $h_a$ and   $\tau_b$ defined   for all $f\in \mathscr P$ by
  \begin{align}
 (h_a f)(x)= f(ax) \ \ \mbox{and}\ \ (\tau_b f)(x)= f(x-b),\quad a\in \C^{\ast}:=\C\backslash\{0\}, \ b\in\C.\label{eq1.4}
   \end{align}
   On the other hand, for any functional $u$, we can write by transposition
\begin{align}
\hskip1cm &\left<\tau_{-b}u,f(x)\right>:=\left<u,\tau_{b}f(x)\right>=\left<u,f(x-b)\right>,&& f\in {\mathscr P},\label{eq1.5}\\
&\left<h_{a}u,f(x)\right>:=\left<u,h_{a}f(x)\right>=\left<u,f(ax)\right>,&& f\in {\mathscr P}.\label{eq1.6}\hskip1cm
\end{align}
 For further formulas and other properties fulfilled by the operator $D_\omega$  see \cite{AbMa}.\\
We now consider the sequence of monic polynomials $\{Q_n(x):=(n+1)^{-1}D_\omega P_{n+1}(x)\}_{n\geqslant0}$, with its associated dual sequence denoted by $\{v_n\}_{n\geqslant0}$ and fulfilling
\begin{align}
D_{-\omega}\left(v_n\right)=-(n+1)u_{n+1},\ n\geqslant0,\label{eq1.7}
\end{align}
where by definition
\begin{align}
\big< D_{-\omega}u\, ,\, f\big>=-\big< u\, ,\, D_\omega f\big> ,\, u\in \mathscr P' ,\, f\in \mathscr P.\label{eq1.8}
\end{align}
Next, in the light of the so-called Hahn property \cite{Hahn1}, we give the following definition.
\begin{definition}
The {OPS} $\{P_n\}_{n\geqslant0}$ is called ``$D_\omega$-classical” if the sequence of its derivatives $\{Q_n\}_{n\geqslant 0}$ is also a {OPS}.
\end{definition}
Thus,  $\{P_n\}_{n\geqslant0}$ is orthogonal w.r.t. $u_0$ and satisfies the second-order recurrence relation
\begin{subequations}
\begin{align}
&P_{n+2}(x) = (x - \beta_{n+1})P_{n+1}(x)-\gamma_{n+1}P_n (x), \ n\geqslant0,\label{1.9a}\\
&P_1(x) = x - \beta_0,\ P_0(x) = 1,\label{1.9b}
\end{align}
\end{subequations}
and $\{Q_n\}_{n\geqslant0}$ is orthogonal w.r.t. $v_0$  and satisfies the second-order recurrence relation
\begin{subequations}
\begin{align}
&Q_{n+2}(x) = (x - \tilde\beta_{n+1})Q_{n+1}(x)  - \tilde\gamma_{n+1} Q_n(x), \ n\geqslant0,\label{1.10a}\\
& Q_1(x) = x - \tilde\beta_0 ,\ Q_0(x) = 1, \label{1.10b}
\end{align}
\end{subequations}
with the regularity conditions $\gamma_{n} \ne0$ and $\tilde\gamma_{n} \ne0$ for every $n\geqslant1$.\\
Finally, we will summarise in the following proposition the important properties characterizing the $D_\omega$-classical orthogonal polynomials as 
stated in \cite[Propositions 2.1-2.3]{AbMa}.
\begin{proposition}
For any {\rm OPS} $\{P_n\}_{n\geqslant0}$, the following are equivalent statements:\\
{\rm(a)} The sequence $\{P_n\}_{n\geqslant0}$ is $D_\omega$-classical.\\
{\rm(b)} The sequence $\{Q_n\}_{n\geqslant0}$ is orthogonal.\\
{\rm(c)} There exist two polynomials $\Phi$ monic $(with\, \deg \Phi= t\leqslant2)$
 and $ \Psi\ (with\, \deg \Psi=1)$, and a sequence $\{\lambda_n\}_{n\geqslant0}$, $\lambda_n\ne0$ for all $n$, such that
  \begin{align}
   \Phi(x)\left(D_\omega D_{-\omega}P_{n+1}\right)(x)-\Psi(x)\left(D_{-\omega}P_{n+1}\right)(x)+\lambda_nP_{n+1}(x)=0,\ n\geqslant0.\label{eq1.11}
   \end{align}
{\rm(d)} The sequences $\{Q_n\}_{n\geqslant0}$ and $\{P_n\}_{n\geqslant0}$ are interlinked via the differentiation formula
 \begin{align}
  \Phi(x)Q_n(x)=\alpha_{n+2}^2P_{n+2}(x)+\alpha_{n+1}^1P_{n+1}(x)+\alpha_{n}^0P_{n}(x),\ n\geqslant0,\quad(\alpha_{n}^0\ne0).\label{eq1.12}
 \end{align}
 This identity is referred to as the {\it first structure relation} of the OPS $\{P_n\}_{n\geqslant0}$.\\
{\rm(e)}The form $u_0$ is $D_\omega$-classical,  say, it is regular and satisfies the functional equation
  \begin{align}
  D_{-\omega}\left(\Phi u_0\right)+\Psi u_0=0.\label{eq1.13}
   \end{align}
   {\rm(f)} There exist a monic polynomial $\Phi$,  $ \deg \Phi\leqslant2$, and a sequence $\{\lambda_n\}_{n\geqslant0}$, $\lambda_n\ne0$ for all $n$,
   such that the canonical form $u_0$ satisfies the so-called functional Rodrigues formula
     \begin{align}
P_n u_0 = \lambda_n D^n_{-\omega}\Big\{\Big(\prod_{\nu=0}^{n-1}\tau_{-\nu \omega}\Phi \Big)u_0\Big\} ,\ n\geqslant0,\quad
 \mbox{ with } \prod_{\nu=0}^{-1}\!:=1.\label{eq1.14}
   \end{align}
\end{proposition}
Here we will add a new characterization to those established in this proposition. This will be proved once we give the lemma below. 
To begin with, apply the operator $D_\omega$ to \eqref{1.9a}-\eqref{1.9b}, taking into account \eqref{1.10a}-\eqref{1.10b}, we obtain
\begin{subequations}
\begin{align}
P_{n+2}(x) &= Q_{n+2}(x) + {\tilde\alpha_{n+1}}^1Q_{n+1}(x) + {\tilde\alpha_{n}}^0 Q_{n} (x),\ n\geqslant0, \label{1.15a}\\
 P_1(x) &= Q_1(x) +\tilde\alpha_0^1 ,\; P_0(x) = Q_0(x) = 1.\label{1.15b}
 \end{align}
 \end{subequations}
 We will refer to this relation as the {\it second structure relation} of the polynomials $P_n,\,n\geqslant0$.
 This will be instrumental in Section 2 to derive the  system connecting the coefficients $\alpha_{n+2}^2$, $\alpha_{n+1}^1$ and $\alpha_n^0$,
for $n\geqslant0$, and the recurrence coefficients $\beta_n$ and $\gamma_{n+1},\ n\geqslant0$.\\
\noindent Combining \eqref{1.15a}-\eqref{1.15b} with \eqref{1.9a}-\eqref{1.9b} and  then use \eqref{1.10a}-\eqref{1.10b}, we infer that
\begin{align}
\tilde\alpha_n^1=(n+1)(\beta_{n+1}-\tilde\beta_{n}-\omega),\ n\geqslant0,\ \ \mbox{and}\ \ \tilde\alpha_n^0=(n+1)\gamma_{n+2}-(n+2)\tilde\gamma_{n+1},\
 n\geqslant0,\label{eq1.16}
\end{align}
\noindent Reminder that, for the classical orthogonal polynomials ($\omega=0$), the  first structure relation was given by Al Salam and Chihara 
\cite{AlCh}, and the second one was established by Maroni \cite{Maro3}.\par
\noindent When $D_\omega$ is replaced by the finite difference operator $\Delta$, Garcia {\it et al.} \cite{GaMaSa} proved that the structure relations \eqref{eq1.12} and \eqref{1.15a}-\eqref{1.15b}, as well as the functional Rodrigues formula \eqref{eq1.14} characterize the discrete classical polynomials
 of Charlier, Meixner, Krawchuk and Hahn.\par
\noindent We will now return to the operator $D_{\omega}$ and show that \eqref{1.15a}-\eqref{1.15b} also characterize the $D_\omega$-classical orthogonal polynomials. To do so, we need the following lemma.\par
\begin{lemma}[\cite{Maro}]
Let  $\{P_n\}_{n \geqslant 0}$ be a sequence of monic polynomials and let $\{u_n\}_{n \geqslant 0}$ be its associated dual sequence.
For any linear functional $u$ and integer $ m\geqslant 1$, the following statements are equivalent:\\
\noindent ${\rm (i)} \ \ \big< u \, ,\, P_{m-1} \big> \not = 0 \ ;\  \big< u \, ,\, P_{n} \big> = 0 ,\; n\geqslant m;$\\
\noindent ${\rm (ii)}\ \  \exists \ \lambda_\nu \in \mathbb{C} ,\ 0 \leqslant \nu \leqslant m-1 , \ \lambda_{m-1}\not= 0 ,\
\mbox{ such that } u = \sum^{m-1}_{\nu = 0} \lambda_\nu {\it u}_\nu. $
 \end{lemma}
 \begin{proposition}
Let $\{P_n\}_{n\geqslant0}$ be an {\rm OPS} satisfying \eqref{1.9a}-\eqref{1.9b}. The sequence $\{P_n\}_{n\geqslant0}$ is $D_\omega$-classical
if and only if it fulfils \eqref{1.15a}-\eqref{1.15b}.
 \end{proposition}
  \begin{proof} The proof is similar in spirit to that of \cite[Proposition 2.10]{GaMaSa}.  The necessary condition has already been shown.
  Conversely, suppose that the OPS $\{P_n\}_{n\geqslant0}$ fulfils \eqref{1.15a} with \eqref{1.15b}. \\
  The action of the functional $v_0$ on both sides of the aforementioned identities gives rise to
\[
\big<v_0 \, ,\, P_0 \big>  =1 \ ,\ \big<v_0 \, ,\, P_1 \big>=\tilde\alpha_0^1\ ,\ \big<v_0 \, ,\, P_2 \big>=\tilde\alpha_0^0 \ \ \mbox{and} \ \
\big<v_0 \, ,\, P_n \big>=0 ,\ \ \mbox{for}\ \ n\geqslant3.
\]
Application of Lemma 1.3 shows that
\begin{align}
 v_0=\lambda_0u_0+\lambda_1u_1+\lambda_2u_2,\label{eq1.17}
 \end{align}
with $\lambda_0=1$, $\lambda_1=\tilde\alpha_0^1=\beta_{1}-\tilde\beta_{0}-\omega$ and $\lambda_2=\tilde\alpha_0^0=\gamma_{2}-2\tilde\gamma_{1}$.\\
Now, the use of \eqref{eq1.3} enables us to write $ u_1=\left(\big<u_0,P^2_1\big>\right)^{-1}P_1u_0$ and $u_2=\left(\big<u_0,P^2_2\big>\right)^{-1}P_2u_0$.\\
In \eqref{eq1.17} we replace $u_1$ and $u_2$ by their respective expressions given above, to deduce that there exists a polynomial $\Phi$, with $\deg\Phi\leqslant2$, such that $v_0=\Phi u_0$.\\
On the other hand, setting  $n=0$ in \eqref{eq1.7}, it follows immediately that
\[
D_{-\omega}\left(v_0\right)=-u_1=-\left(\big<u_0,P^2_1\big>\right)^{-1}P_1u_0:=-\Psi u_0.
\]
Combining these last results, we deduce that
\[
 D_{-\omega}\left(\Phi u_0\right)+\Psi u_0=0,\ \ \mbox{with} \ \ \deg\Phi\leqslant2 \ \ \mbox{and} \ \ \deg\Psi=1.
 \]
By Proposition 1.2, we easily conclude that the orthogonal polynomials sequence $\{P_n\}_{n\geqslant0}$ is $D_\omega$-classical, and the proof is complete.
\end{proof}
 \smallskip
\noindent  At the end  of this section, let us remember the definition of the {\it shifted} polynomials denoted $\{\hat{P}_n\}_{n\geqslant0}$
corresponding to the PS $\{P_n\}_{n\geqslant0}$. For all $n, n=0, 1,\ldots,$ we have
\begin{align}
{\hat P}_n(x) :=\hat{a}^{-n} P_n(\hat{a}x+\hat{b}),\ \hbox{\rm for}\ (\hat{a} ; \hat{b})\in \C^\ast\times\C.\label{eq1.18}
\end{align}
Since the classical character of the considered polynomials is preserved by any linear change of the variable, for the OPS $\{P_n\}_{n\geqslant0}$
satisfying \eqref{1.9a}-\eqref{1.9b}, we obtain that the polynomials ${\hat P}_n, {n=0, 1,\ldots}$, satisfy also the second-order recurrence relation
 \begin{subequations}
\begin{align}
&{{\hat P}}_{n+2}(x) = (x - {\hat\beta}_{n+1}){{\hat P}}_{n+1}(x)-{\hat\gamma}_{n+1}{{\hat P}}_n (x), \ n\geqslant0,\label{1.19a}\\
&{{\hat P}}_1(x) = x - {\hat\beta}_0,\ {{\hat P}}_0(x) = 1,\label{1.19b}
\end{align}
 \end{subequations}
with
\begin{align}
{\hat\beta}_{n}=\frac{{\beta}_{n}-\hat{b}}{\hat{a}},\, n\geqslant0,\ \  \mbox{and}\ \ {\hat\gamma}_{n+1}=\frac{{\gamma}_{n+1}}{\hat{a}^2}, \,
 n\geqslant0\quad (\hat{a}\ne0).\label{eq1.20}
\end{align}
Furthermore, from \eqref{eq1.5}-\eqref{eq1.6} we readily see that \eqref{eq1.18} becomes
 $\hat{P}_n(x) ={\hat{a}}^{-n}\big(h_{\hat{a}}\circ\tau_{-\hat{b}}P_n\big)(x)$.\\
In addition, if $u_0$ is $D_\omega$-classical, then Equation \eqref{eq1.13} leads to
\begin{align}
D_{-\frac{\omega}{\hat{a}}}\big(\hat{\Phi}\hat{u}_0\big)+\hat{\Psi} \hat{u}_0=0,\label{eq1.21}
\end{align}
where $\hat{\Phi}(x)=\hat{a}^{-t}\Phi(\hat{a}x+\hat{b})$, $\hat{\Psi}=\hat{a}^{1-t}\Psi(\hat{a}x+\hat{b})$ and
 $\hat{u}_0=(h_{\hat{a}^{-1}}\circ\tau_{-\hat{b}})u_0$.\par
\smallskip
\noindent The  paper is organized as follows. In the next section we pose and solve two nonlinear systems. The first and most fundamental one relates the recurrence coefficients $\beta_n$,  $\gamma_{n+1}$ with  $\tilde\beta_n$, $\tilde\gamma_{n+1}$.
The second system combines the coefficients $\alpha_n^i,\ i=0, 1, 2\ ;\ \tilde\alpha^j_n,\ j=0, 1$, $\tilde\beta_n$ and $\tilde\gamma_{n+1}$ with those of the polynomial $\Phi$ (Proposition 2.2). This allows to express the coefficients $\alpha_n^i$, in terms of $\beta_n$ or $\gamma_{n+1}$.
In Section 3, we investigate the canonical families of $D_\omega$-classical orthogonal polynomials which we identify after assigning particular values
 to the free parameters.
 The last section is devoted to the sequences of higher order derivatives. We give, principally, the explicit expressions of their recurrence coefficients
  in terms of the coefficients $\big(\beta_n, \gamma_{n+1}\big)_{n\in\N}$. When $\omega=0$, we rediscover the link between every higher order derivative sequences for the classical polynomials of Hermite, Laguerre, Bessel and Jacobi with each of these families.
\section{Computation of the  related coefficients}
In order to compute the various related recurrence coefficients, the first step is to establish the main system connecting the recurrence coefficients
$ \{\beta_n\}_{n\geqslant0} , \, \{\gamma_{n+1}\}_{n\geqslant0}$ with  $\{\tilde\beta_n\}_{n\geqslant0} $, $ \{\tilde\gamma_{n+1}\}_{n\geqslant0}$.
To do this, we proceed as follows. Substituting  in \eqref{1.9a}, $P_{n+2}, P_{n+1}$ and $ P_n$ by their expressions provided in \eqref{1.15a}
we derive a relation in terms of the polynomials $Q_k$ for $ k=n+2,\ldots, n-2$. Now, in this new  relation we replace  $xQ_{n+1}$, $xQ_n$
 and $xQ_{n-1}$ by their respective expressions derived from the recurrence relation \eqref{1.10a} obtaining an expansion depending
 only on the polynomials $Q_{n+2},\ldots, Q_{n-2}$.\\
 After rearranging the terms in the resulting expansion and making some simplifications, the next system (valid for all $n\geqslant1$) follows by identification
   \begin{align*}
&(n+2) \tilde\beta_{n} - n \tilde\beta_{n-1}= (n+1)\beta_{n+1} -(n-1)\beta_{n}-\omega;\\
&2\tilde\beta_0 = \beta_1 +\beta_0 -\omega,\\
&(n\!+\!3) \tilde\gamma_{n+1} -\! (n\!+\!1) \tilde\gamma_n\! =\! (n\!+\!1)\gamma_{n+2}\!-\!(n\!-\!1) \gamma_{n+1}\! +\! (n\!+\!1)\big(\beta_{n+1}\!
-\! \tilde\beta_n\big)\big(\beta_{n+1}\! - \!\tilde\beta_n-\omega\big);\\
&3 \tilde\gamma_1 = \gamma_2 + \gamma_1 + \big(\beta_1 -\tilde\beta_0\big)\big(\beta_1 -\tilde\beta_0-\omega\big),\\
& (n+1)\tilde\gamma_{n} \big(2\beta_{n+1} - \tilde\beta_{n} -\tilde\beta_{n-1}-\omega\big)
-n \gamma_{n+1}\big(\beta_{n+1} + \beta_{n} - 2 \tilde\beta_{n-1}-\omega\big)=0,\\
&(n+2)\tilde\gamma_n\tilde\gamma_{n+1} -2(n+1)\tilde\gamma_n \gamma_{n+2}+n \gamma_{n+1} \gamma_{n+2}=0.
\end{align*}
We should mention here the important role played by the Hahn property (Definition 1.1) to establish such a system, since we have used only the fact that these polynomials as well as their $D_\omega$-derivatives are orthogonal w.r.t. regular forms. In other words, each sequence satisfies a second order recurrence relation, as it is shown in Sec. 1.\\
 When $\omega=0$, we recover the system initiated and solved by Douak and Maroni \cite{DoMa1} whose the solutions provide the classical OPS of Hermite, Laguerre, Jacobi and Bessel, after assigning particular values to the free parameters. We will encounter these families again  in this paper.\par
 \medskip
 \noindent To solve the above system, let us introduce  the auxiliary coefficients $\delta_n$ and $\theta_n$ by writing
\begin{align}
\tilde\beta_n &= \beta_{n+1} + \delta_n \, ,\; n\geqslant 0,\label{eq2.1}\\
{\tilde\gamma_n} &= \frac{n}{n+1}\gamma_{n+1}\theta_n\, ,\; n\geqslant 1, \quad\big(\theta_n\not=0\big).\label{eq2.2}
\end{align}
With these considerations, the two qualities \eqref{eq1.16} take the form
\begin{align}
\tilde\alpha_{n}^1=-(n+1)(\delta_n+\omega),\, \ n\geqslant0 \, ; \ \ \tilde\alpha_{n}^0=(n+1)\gamma_{n+2}(1-\theta_{n+1}),\, \ n\geqslant0.\label{eq2.3}
\end{align}
  \subsection{The coefficients of the recurrence relations}
Our main objective here is to initially compute the auxiliary coefficients $\delta_n$ and $\theta_n$, and then give the explicit expressions of
the  coefficients $\beta_n$ and $\gamma_n$.
This in turn allows to determine the coefficients $\tilde\beta_n$ and $\tilde\gamma_n$ and write significantly better each of the coefficients
$\alpha_{n}^i$ and $\tilde\alpha_{n}^j$.\\
Under the formulas \eqref{eq2.1}-\eqref{eq2.2}, it is easy to see that the above system can be transformed into
\begin{align}
&\beta_{n+1} - \beta_{n} = n\delta_{n-1} - (n+2) \delta_{n}-\omega,\ n\geqslant0,\quad(\delta_{-1}=0),\label{eq2.4}\\
&\big[(n+3) (\theta_{n+1}-1)+1\big]\frac{\gamma_{n+2}}{n+2}-\big[n(\theta_n -1)+1\big]\frac{\gamma_{n+1}}{n+1}=\delta_n(\delta_n+\omega),\ n\geqslant1,
\label{eq2.5}\\
&\big(3 \theta_1 - 2) {\gamma_2} - 2\gamma_1 = 2\delta_0(\delta_0+\omega),\label{eq2.6}\\
&\big[(n+3)(\theta_{n}-1)+1\big]\delta_{n} - \big[(n-1)(\theta_{n}-1)+1\big]\delta_{n-1}+2(\theta_{n}-1)\omega=0,\ n\geqslant1,\label{eq2.7}\\
&\left(\theta_{n+1}-2\right)\theta_n +1=0,\ n\geqslant1,\label{eq2.8}
\end{align}
  To solve this system, we begin with the Riccati equation \eqref{eq2.8} whose solutions are
\begin{align*}
{\bf A.}\quad\theta_n &= 1,\ \ n\geqslant1,\\
{\bf B.}\quad\theta_n &= \frac{n + \theta + 1}{n +\theta},\ \ n\geqslant1,\quad \theta \ne - 1, -2, \ldots.\hskip5cm
\end{align*}
Hence, the first three equations must be examined in the light of these solutions.\par
\medskip
\noindent{\bf Case A.} For $\theta_n = 1$, the above system reduces to
\begin{align}
&\beta_{n+1} - \beta_{n} = n\delta_{n-1} - (n+2) \delta_{n}-\omega,\ n\geqslant0,\label{eq2.9}\\
& \frac{\gamma_{n+2}}{n+2}- \frac{\gamma_{n+1}}{n+1} =\delta_n(\delta_n+\omega),\ n\geqslant0,\label{eq2.10}\\
&\delta_{n+1} - \delta_{n}=0,\ n\geqslant0.\label{eq2.11}
\end{align}
Equation \eqref{eq2.11} clearly shows that $\delta_n=\delta_0,\ n\geqslant0$, giving
 $\delta_n(\delta_n+\omega)=\delta_0(\delta_0+\omega),\ n\geqslant0$.\\
 Thus it is quite natural to single out the  two statements  $\delta_0(\delta_0+\omega)=0$ and $\delta_0(\delta_0+\omega)\ne0$.
 But right now we go back to the two first equations from which we readily deduce that
 \begin{align}
 &\beta_n=\beta_0-(2\delta_0+\omega)n,\ n\geqslant0,\label{eq2.12}\\
 &\gamma_{n+1}=(n+1)\big(\delta_0(\delta_0+\omega)n+\gamma_1\big),\ n\geqslant0.\label{eq2.13}
 \end{align}
If we take $\omega=0$, we recover the recurrence coefficients of the Hermite or Laguerre polynomials as shown in \cite{DoMa1}.
This will be made  more precise in the subcases $\bf A_1$ and $\bf A_2$ below.\par
 \medskip
\noindent {\bf Case B.} For $ \theta_n = {(n + \theta + 1)}/{(n +\theta)}$, the  system \eqref{eq2.4}-\eqref{eq2.8} becomes
\begin{align}
&\beta_{n+1} - \beta_{n} = n\delta_{n-1} - (n+2) \delta_{n}-\omega,\ n\geqslant0,\label{eq2.14}\\
&\frac{(2n+\theta+4)}{(n+\theta+1)} \frac{\gamma_{n+2}}{n+2}-\frac{(2n+\theta)}{(n+\theta)}\frac{\gamma_{n+1}}{n+1} = \delta_n(\delta_n+\omega),
\ n\geqslant0,\label{eq2.15}\\
&\big(2n+\theta+3\big)\delta_{n} - \big(2n+\theta-1\big)\delta_{n-1}=-2\omega,\ n\geqslant1,\label{eq2.16}
\end{align}
 unless, of course,  $\theta$ happen to be zero for the index $n = 0$. So, we will first discuss the solution of the above  system when  $\theta \ne 0$. 
 The  case $\theta = 0$ is special, it will be considered separately. \\
When $\omega=0$, with  appropriate choices of the parameters, the only orthogonal polynomials obtained as solutions of this problem are those of Bessel
 and Jacobi  (see \cite{DoMa1} for more details).\par
\smallskip
\noindent We now return to seeking solutions for the equations \eqref{eq2.14}--\eqref{eq2.16}. Observe first that the RHS of Equation \eqref{eq2.15}  vanishes if and only if $\delta_n=-\omega$ or $\delta_n=0$. Each of these is possible.\\
It is easy to check that the former statement  is dismissed, since it contradicts Equality \eqref{eq2.16}. For the latter, if we replace $\delta_n=0$
in \eqref{eq2.16}, we immediately see that this leads to $\omega=0$.\\
 Straightforwardly from Equations \eqref{eq2.14} and \eqref{eq2.15}, one has
 \begin{align}
 \beta_n=\beta_0\ ; \  \gamma_{n+1}=\gamma_1 \frac{(\theta+2)(n+1)(n+\theta)}{(2n+\theta+2)(2n+\theta)}, \ n\geqslant0.\label{eq2.17}
  \end{align}
If we set $\theta=2\lambda$ and make a linear transformation with $\hat{a}^2=2(\lambda+1)\gamma_1\ ;\ \hat{b}=\beta_0$, then
\begin{align}
 \hat{\beta}_n=0 \ ; \ \hat{\gamma}_{n+1}=\frac{(n+1)(n+2\lambda)}{(n+\lambda+1)(n+\lambda)}, \ n\geqslant0.\label{eq2.18}
 \end{align}
We thus meet the Gegenbauer polynomials which will reappear again in Subcase {$\bf B_{21}$}.\par
\smallskip
\noindent From now on we assume that $\delta_n(\delta_n+\omega)\ne0,\ n\geqslant0.$
Starting from Equation \eqref{eq2.16}, multiply both sides by $2n+\theta+1$, after summation we get
\begin{align}
  \delta_n=\frac{\delta_0(\theta+3)(\theta+1)-2\omega n(n+\theta+2)}{(2n+\theta+3)(2n+\theta+1)},\, n\geqslant0.\label{eq2.19}
\end{align}
Use a division to obtain
\begin{subequations}
\begin{align}
  \delta_{n}&=\frac{2\mu}{\big(2n+\theta+3\big)\big(2n+\theta+1\big)}-\frac{1}{2}\omega,\, n\geqslant0,\label{2.20a}\\
\delta_{n}&=\mu\big(\vartheta_{n}-\vartheta_{n+1}\big)-\frac{1}{2}\omega,\, n\geqslant0,\label{2.20b}
\end{align}
\end{subequations}
where we have written $\mu:=\frac{1}{4}(2\delta_0+\omega)(\theta+3)(\theta+1)$ and $\vartheta_n =\left(2n+\theta+1\right)^{-1}\!,\  n\geqslant 0 $.\\
Thanks to the identity \eqref{2.20b}, we can write
\begin{align}
 \left(2n+\theta+2\right)\delta_n(\delta_n+\omega)= \mu^2\big({\vartheta_{n}^2}-{\vartheta_{n+1}^2}\big)-\frac{1}{4}{\omega^2}(2n+\theta+2),\ n\geqslant 0.
 \label{eq2.21}
 \end{align}
The objective of course is to incorporate this new expression into \eqref{eq2.15} to derive the coefficients $\gamma_n,\ n\geqslant1$, which will in 
fact be processed in a next step. We first calculate the coefficients $\beta_n$. For this, observe that the RHS of \eqref{eq2.14} may be rewritten using 
 Equation \eqref{eq2.16} in the form
 \begin{align}
 n\delta_{n-1} - (n+2)\delta_{n}-\omega=\frac{1}{2}(\theta-1)(\delta_{n} - \delta_{n-1}),\ n\geqslant1.\label{eq2.22}
 \end{align}
 It is easily seen that, if $\theta$ assumes the value $1$, the equation \eqref{eq2.22} provides $\delta_n=-\frac{1}{2}\omega, \, n\geqslant0$.\\
 As a straightforward consequence of this last result,  one sees immediately that \eqref{eq2.14}, \eqref{eq2.15} respectively provides
 \begin{align}
 \beta_{n}&=\beta_0,\ n\geqslant 0,\label{eq2.23}\\
 \gamma_{n+1}&=-\frac{1}{4}\frac{(n+1)^2\left(\omega^2 n(n+2)-12\gamma_1\right)}{(2n+3)(2n+1)},\ n\geqslant0.\label{eq2.24}
 \end{align}
  When $\omega=0$, under the transformation $\hat{a}^2=3\gamma_1, \hat{b}=\beta_0$, we meet the Legendre polynomials.\par
  \medskip
\noindent We now turn to the case $\theta\ne1$. Using the identity \eqref{eq2.22},   Equation \eqref{eq2.14} gives rise to
 \begin{align}
 &\beta_{n+1}-\beta_n=\frac{1}{2}(\theta-1)(\delta_n-\delta_{n-1}),\ n\geqslant 1,\label{eq2.25}\\
  &\beta_1-\beta_0=-(2\delta_0+\omega),\label{eq2.26}
   \end{align}
 with $\delta_n$ is given by \eqref{eq2.19}.  From this, it may be concluded that
   \begin{align}
  \beta_{n}=\beta_0-\frac{(2\delta_0+\omega)(\theta+3)n(n+\theta)}{(2n+\theta+1)(2n+\theta-1)},\ n\geqslant 0.\label{2.27}
 \end{align}
 We can now proceed to compute the coefficients $\gamma_{n+1}, n\geqslant0$. To do so, multiply both sides of Equation \eqref{eq2.15} by $2n+\theta+2$ 
 and set
 \begin{align}
 \Theta_{n+1}=\frac{\left(2n+\theta+2\right)(2n+\theta)}{(n+\theta)}\frac{\gamma_{n+1}}{n+1},\ n\geqslant 0.\label{eq2.28}
 \end{align}
Then, taking into consideration \eqref{eq2.21}, we easily check that \eqref{eq2.15} takes the form
 \begin{align*}
  \Theta_{n+2}- \Theta_{n+1} =\mu^2\big({\vartheta_{n}^2}-{\vartheta_{n+1}^2}\big) -\frac{1}{4}{\omega^2}(2n+\theta+2), \, n\geqslant 0.
\end{align*}
By summation, we deduce that 
 \begin{align}
 \Theta_{n+1}=\Theta_1+\mu^2\big(\vartheta_0^2-\vartheta_n^2\big)-\frac{1}{4}{\omega^2}n(n+\theta+1),\ n\geqslant0.\label{2.29}
 \end{align}
 Substituting \eqref{eq2.28} into \eqref{2.29} yields
\begin{align}
\gamma_{n+1}=-\frac{(n\!+\!1)(n\!+\!\theta)\!\Big\{\!\big[\frac{1}{4}\omega^2n\left(n+\!\theta\!+\!1\right)\!-\!\left(\mu^2\vartheta_0^2\!+\!(\theta\!+\!2)
\gamma_1\right)\!\big](2n+\theta+1)^2\!+\!\mu^2\Big\}}{(2n+\theta+2)(2n+\theta+1)^2(2n+\theta)}.\label{eq2.30}
\end{align}
It is possible to write the expression between braces in the numerator of  \eqref{eq2.30}  in the form
\begin{align*}
\big(\omega n(n+\theta+1)+\varrho n+\rho_1\big)\big(\omega n(n+\theta+1)-\varrho n+\rho_2\big),
\end{align*}
where the three parameters $\rho_1$, $\rho_2$ and $\varrho$ are such that
 \begin{align}
 &(\theta+1)\varrho^2+(\rho_2-\rho_1)\varrho=0,\label{2.31}\\
 &\varrho^2-(\rho_2+\rho_1)\omega=\big((\theta+3)\delta_0+(\theta+2)\omega\big)\big((\theta+3)\delta_0+\omega\big)+4(\theta+2)\gamma_1,\label{eq2.32}\\
 &\rho_2\rho_1=-(\theta+1)^2(\theta+2)\gamma_1.\label{eq2.33}
 \end{align}
The roots of the quadratic equation \eqref{2.31} are $\varrho=0$ and $\varrho={(\rho_1-\rho_2)}/{(\theta+1)}$ for $\rho_2\ne\rho_1$, with the root $\varrho=0$ being double, if $\rho_2=\rho_1$. The last equation clearly shows that $\rho_2\rho_1\ne0$.\\
 All these parameters will be well specified when dealing with the canonical families. But, in any way, we have to consider the following two cases.
\begin{enumerate}[\bf 1.]
\item For $\varrho=0$,  we obtain
\begin{align}
\gamma_{n+1}=-\frac{(n+1)(n+\theta)\big(\omega n(n+\theta+1)+\rho_1\big)\big(\omega n(n+\theta+1)+\rho_2\big)}
{(2n+\theta+2)(2n+\theta+1)^2(2n+\theta)},\ n\geqslant0.\label{eq2.34}
\end{align}
In the particular case $\rho_2=\rho_1:=\rho$, \eqref{eq2.34} simplifies to
\begin{align}
\gamma_{n+1}=-\frac{(n+1)(n+\theta)\big(\omega n(n+\theta+1)+\rho\big)^2}
{(2n+\theta+2)(2n+\theta+1)^2(2n+\theta)},\ n\geqslant0.\label{eq2.35}
\end{align}
\item For the general case $\varrho\ne0$, we have
\begin{align}
\gamma_{n+1}=\!-\frac{(n\!+\!1)(n+\theta)\big(\omega n(n+\!\theta\!+\!1)+\varrho n+\rho_1\big)\big(\omega n(n+\!\theta\!+\!1)-\varrho n+\!\rho_2\big)}
{(2n+\theta+2)(2n+\theta+1)^2(2n+\theta)},\, n\geqslant0.\label{eq2.36}
\end{align}
\end{enumerate}
We now turn to the special case $\theta=0$. Substituting this in \eqref{eq2.4}-\eqref{eq2.7} yields
\begin{align*}
&\beta_{n+1} - \beta_{n} = n\delta_{n-1} - (n+2) \delta_{n}-\omega,\ n\geqslant0,\\
& \gamma_{n+2}- \gamma_{n+1} = \frac{1}{2}(n+1)\delta_n(\delta_n+\omega),\ n\geqslant1,\\
& {\gamma_2} - \frac{1}{2}\gamma_1 = \frac{1}{2}\delta_0(\delta_0+\omega),\\
&(2n+3)\delta_{n} - (2n-1)\delta_{n-1}=-2\omega,\ n\geqslant1.
\end{align*}
The same reasoning applies to this case gives
\begin{align*}
\delta_{n}&= \frac{3\delta_0-2\omega n(n+2)}{(2n+3)(2n+1)},\ n\geqslant 0.\\
\beta_{n}&=\beta_0-\frac{3(2\delta_0+\omega)n^2}{(2n+1)(2n-1)},\ n\geqslant 0.\\
\gamma_{n+1}&=-\frac{\big(\omega n(n+1)+\tau n+\tau_1\big)\big(\omega n(n+1)-\tau n+\tau_2\big)}{4(2n +1)^2},\, n\geqslant1,
 \end{align*}
 where $\tau_1$, $\tau_2$ and $\tau$ are such that
  \begin{align*}
 & \tau^2+(\tau_2-\tau_1)\tau=0,\\
 & \tau^2-(\tau_2+\tau_1)\omega=(3\delta_0+2w)(3\delta_0+\omega)+8\gamma_1,\\
 &\tau_2\tau_1=-2\gamma_1.
 \end{align*}
 When $\omega=0$, if moreover $\delta_0=0$, which we may assume, it follows that
$$ 
\delta_{n}= 0, \  n\geqslant 0, \ \beta_{n}=\beta_0, \  n\geqslant 0, \ \gamma_{n+1}=  \frac{1}{2}\gamma_1,\, n\geqslant1.
$$
Thus,  choosing $\beta_0=0$ and $\gamma_1=\frac{1}{2}$, we meet the Tchebychev  polynomials of the first kind.\par
\medskip
\noindent After having finished solving the first system, we now proceed to the determination of the  coefficients $\tilde\alpha_{n}^j$ and $\alpha_{n}^j$ 
in terms of $\beta_{n}$ and $\gamma_{n+1},\, n\geqslant0$. This will be done in the next subsection.
\subsection{The coefficients of the structure relations}
\begin{proposition}
Let $ \Phi$ be the  polynomial arising in {\rm Proposition $1.2$}. We let the degree of $ \Phi$ to be two $2$ and write $ \Phi(x)=a_2x^2+a_1x+a_0$.
 Then the coefficients implicated in the two structure relations \eqref{eq1.12} and \eqref{1.15a}-\eqref{1.15b}  are interlinked through the following system

\begin{align}
&\alpha_{n+2}^2=a_2,\ n\geqslant0,\label{eq2.37}\\
&\alpha_{n+1}^1 +a_2\tilde\alpha_{n+1}^1=a_2(\tilde\beta_{n+1}+\tilde\beta_n)+a_1,\ n\geqslant0,\label{eq2.38}\\
&\alpha_{n+1}^1\tilde\alpha_{n}^1 +a_2\tilde\alpha_{n}^0+\alpha_{n}^0=a_2(\tilde\gamma_{n+1}+\tilde\gamma_{n}+\tilde\beta_n^2)+a_1\tilde\beta_n+a_0,
\ n\geqslant0,\label{eq2.39}\\
&\alpha_{n+1}^1\tilde\alpha_{n-1}^0+\alpha_{n}^0\tilde\alpha_{n-1}^1=
a_2\tilde\gamma_{n}(\tilde\beta_{n}+\tilde\beta_{n-1})+a_1\tilde\gamma_{n},\ n\geqslant1,\label{eq2.40}\\
&\alpha_{n+1}^0\tilde\alpha_{n-1}^0=a_2\tilde\gamma_{n+1}\tilde\gamma_{n}, \ n\geqslant1,\label{eq2.41}
\end{align}
where we have adopted  the convention  that $\tilde{\gamma}_0:=0$ so that \eqref{eq2.39} remains valid for $n = 0$.
\end{proposition}
\begin{proof}
 As for the preceding system, we give only the main ideas of the proof. First, comparison of coefficients in \eqref{eq1.12} shows that $\alpha_{n+2}^2=a_2$.
 Now, use the recurrence relation \eqref{1.10a} twice to write the product $\Phi(x) Q_{n}$, which is the LHS of \eqref{eq1.12}, in terms of the polynomials $Q_{n+2},\ldots, Q_{n-2}$.
Then replace in the RHS of \eqref{eq1.12} $P_{n+2}, P_{n+1}$ and $P_{n}$ by their expressions provided in \eqref{1.15a} to obtain another expansion in
terms of the polynomials $Q_{n+2},\ldots, Q_{n-2}$. After some simplifications, and by identification the equations \eqref{eq2.38}-\eqref{eq2.41} follow.
\end{proof}\par
\noindent As far as the author knows, the technique used here to find  explicit expressions for the coefficients involving in \eqref{1.15a} with \eqref{1.15b} is new, and the results obtained still unknown.
So, the solution of the above system brings  an answer to this question. But before doing so, recall that when the form $u_0$ is $D_\omega$-classical,
namely, it is regular and satisfies \eqref{eq1.13}, the polynomials $ \Phi$ and $ \Psi$ necessarily satisfy (see \cite[p.7]{AbMa} for further details):
\begin{align}
\kappa \Phi(x) &=(1\!-\!\theta_1)x^2-\!\bigl((1\!-\!\theta_1)(\beta_1+\beta_0)+\delta_0+\omega\bigr)x+
\bigl((1\!-\!\theta_1)\beta_1+\delta_0+\omega\bigr)\beta_0+\theta_1\gamma_1, \label{eq2.42}\\
\kappa \Psi(x) &=P_1(x)=x-\beta_0.\label{eq2.43}
\end{align}
Note that the expression in the RHS of \eqref{eq2.42} is slightly transformed in accordance with the relations \eqref{eq2.1} and \eqref{eq2.2}.
The coefficient $\kappa$ is to be chosen later so that $\Phi(x)$ is being monic.\par
\smallskip
\noindent Since we are proceeding following the two situations {Case \bf A}  and {Case \bf B}, we see that the leading coefficient  $a_2$ of $\Phi(x)$ assumes the value $0$, when $\theta_1=1$, or is such that $\kappa a_2= -(\theta+1)^{-1}$, when $\theta_1=(\theta+2)/(\theta+1)$. To get $a_2=1$ in the 
latter case, we  choose $\kappa = -(\theta+1)^{-1}$ which, in turn, allows to write the polynomial $\Phi(x)$ in one of the four standard forms
\[
\Phi(x)=1,\ \Phi(x)=x,\ \Phi(x)= x^2\ \mbox{and}\ \Phi(x)=(x+1)(x-c),\ c\in \C\backslash\{-1\}.
\]
The classification achieved according to the degree of $\Phi$ as in \cite{AbMa} is of course  exhaustive and is equivalent to that  based on the values 
of $\theta_n$. It is this second alternative that we will retain in the next section to go over the diverse families of $D_\omega$-classical orthogonal polynomials or some relevant cases. But before doing so, let us discuss the solutions of the system \eqref{eq2.37}-\eqref{eq2.41}  when $a_2$ takes one
 of the values  $0$ or $1$, with use of \eqref{eq2.3}.\par
\smallskip
 \noindent{\bf I.}  For $a_2=0$, since $\alpha_{n}^0\ne0$ for each $n$, Equation \eqref{eq2.41} readily gives $\tilde\alpha_{n}^0=0, n\geqslant0$,
  so, due to \eqref{eq2.3}, we necessarily have $\theta_{n}=1$ for all $n\geqslant1$. It turns out that $\delta_n=\delta_0, n\geqslant0$, 
  and so $\tilde\alpha_{n}^1=-(n+1)(\delta_0+\omega), n\geqslant0$. In this case, it is easily seen that the coefficients $\beta_n$ and $\gamma_{n+1}$
   are  given by \eqref{eq2.12}-\eqref{eq2.13}.\\
This  actually happens in the case {\bf A} when the polynomial  $\Phi$ takes the form
\begin{align}
\kappa \Phi(x) =-(\delta_0+\omega)x+(\delta_0+\omega)\beta_0+\gamma_1.\label{eq2.44}
  \end{align}
Using the fact that the polynomial $\Phi(x)$ is monic, we are brought to consider the following two situations according with the degree of this polynomial.
  \newcounter{saveenum}
  \begin{enumerate}[{\bf (i)}]
\item $\Phi(x)$ is constant. In this case we have $\delta_0+\omega=0$ which leads to $a_2=a_1=0$ and $a_0=1$. \\
It follows that $\kappa=\gamma_1$, $\Phi(x)=1$ and $\Psi(x)=\gamma_1^{-1}P_1(x)$. Therefore $\tilde\alpha_{n}^1=0,\ n\geqslant0$,
 and  the above system readily gives
$ \alpha_{n+2}^2=\alpha_{n+1}^1=0\ ; \ \alpha_{n}^0=1, \, n\geqslant0.$\\
Thus, $Q_n=P_n,\, n\geqslant0$, and so the two structure relations  coincide.
 \item $ \Phi(x)$ is linear.  By setting $\kappa=-(\delta_0+\omega)\ne0$ and  $(\delta_0+\omega)\beta_0+\gamma_1=0$, we conclude that
 $a_2=a_0=0$ and $a_1=1$  and so $\Phi(x)=x$ and $\Psi(x)=\gamma_1^{-1}\beta_0P_1(x)$. We thus have
\setcounter{saveenum}{\value{enumi}}
$$
\alpha_{n+2}^2=0 \ ; \ \alpha_{n+1}^1=1\ \hbox{and} \ \alpha_{n}^0=\beta_0-{\delta}_0{n},\, n\geqslant0.
$$
 Accordingly, the two structure relations may be written as follows
 \begin{align*}
 xQ_n(x)&=P_{n+1}(x)+(\beta_0-\delta_0 n)P_{n}(x),\, n\geqslant0,\\
 P_n(x)&=Q_n(x)-(\delta_0+\omega)n Q_{n-1}(x),\, n\geqslant0,\  (Q_{-1}:=0).
 \end{align*}
For $n=1$ in the first relation, we recover the equality $(\delta_0+\omega)\beta_0+\gamma_1=0$, that is, $\kappa\beta_0=\gamma_1$.
This interconnection between  $\beta_0$ and $\gamma_1$ will prove useful in Subcase ${\bf A_2}$.
  \end{enumerate}\par
 \noindent{\bf II.} For $a_2\ne0$, $\Phi(x)$ is then quadratic and so $\kappa=1\!-\theta_1\!=-(\theta+1)^{-1}$, since we take $a_2$ to be $1$. \\
Changing $n$ into $n+1$ in \eqref{eq2.41} yields  $\alpha_{n+2}^0\tilde\alpha_{n}^0\ne0, \, n\geqslant0$, from which we see that the coefficients
  $\displaystyle\tilde\alpha_{n}^0$ are not identically $0$ for all $n$, and hence, due to \eqref{eq2.3}, we conclude that $\theta_{n}\ne1, \ n\geqslant1$.
   This in fact shows that the case {\bf B} is the one to be naturally considered here. We thus have
$$
\tilde\alpha_{n}^1=-(n+1)(\delta_n+\omega)\ \ \mbox{and}\ \ \tilde\alpha_{n}^0=-\frac{n+1}{n+\theta+1}\gamma_{n+2}, \ n\geqslant0.
$$
 Additionally,  the coefficients $\beta_n$ are given by \eqref{2.31}, while the $\gamma_{n}$  are generated either by \eqref{eq2.34} or by \eqref{eq2.36}.
Moreover, taking into account the position of the zeros of the polynomial $\Phi$, we have to consider again two subcases .
\begin{enumerate}[{\bf (i)}]
\setcounter{enumi}{\value{saveenum}}
\item $\Phi(x)=x^2$. Since $a_2=1$ and  $a_1= a_0=0$, \eqref{eq2.42} leads to $\beta_1+\beta_0=(\theta+1)(\delta_0+\omega)$ and $\beta_0^2=-(\theta+2)\gamma_1$.
Analogously to the former case, taking into consideration \eqref{eq2.3},  the system  \eqref{eq2.37}-\eqref{eq2.41} readily provides
\begin{align*}
\alpha_{n+2}^2=1\ ;\ \alpha_{n+1}^1=\beta_{n+1}+\beta_n+n(\delta_{n-1}+\omega)\ \ \mbox{and}\ \ \alpha_{n}^0=-\frac{n+\theta+1}{n+1}\gamma_{n+1},\,
 n\geqslant0.
 \end{align*}
\item $\Phi(x)=(x+1)(x-c),\ c\ne-1$. We  have $a_2=1$, $a_1=1-c$ and   $a_0=-c$.
Therefore, \eqref{eq2.42} gives rise to
 $\beta_1+\beta_0=(\theta+1)(\delta_0+\omega)+c-1$ and $(\beta_0+1)(\beta_0-c)=-(\theta+2)\gamma_1$.
In the same manner we can see that
\begin{align*}
 \alpha_{n+2}^2=1\ ; \ \alpha_{n+1}^1=\beta_{n+1}+\beta_n+n(\delta_{n-1}+\omega)-c+1\ \hbox{and} \ \alpha_{n}^0=-\frac{n+\theta+1}{n+1}\gamma_{n+1},
 \, n\geqslant0.
 \end{align*}
\end{enumerate}
 \section{The canonical families of $D_\omega$-classical  polynomials}
In order to present an exhaustive classification of those polynomials, we will examine different situations in both cases {\bf A} and {\bf B}. Under 
certain restrictions  on the  parameters, we rediscover the well-known families of discrete classical orthogonal polynomials or some particular cases.\par
 \noindent Since we are only interested in regular OPS, the  finite sequences are not considered here.\\
 We often use the linear transformation \eqref{eq1.18} with \eqref{eq1.20}  to provide  the desired results. This is also achieved on account of the 
 specific conditions observed in Subsection 2.2. For each situation, we summarise the relevant properties of the corresponding family of polynomials.\par
 \medskip
 \noindent  $\triangleright$ {\bf Case A.} We first investigate the two main subcases, namely, $\delta_0+\omega=0$ and $\delta_0+\omega\ne0$. Then, we consider the particular subcase when $\delta_0$ assumes the value $0$. \par
    \smallskip
 \noindent{$ \bf A_1:$} $\delta_0+\omega=0$. From \eqref{eq2.12}-\eqref{eq2.13}, taking into account \eqref{eq2.1}-\eqref{eq2.2}, we immediately obtain
 $\tilde\beta_n=\beta_n=\beta_0+\omega n\ ; \ \tilde\gamma_{n+1}=\gamma_{n+1}=\gamma_1(n+1),\  n\geqslant0$. It follows that $Q_n=P_n,\, n\geqslant0$,
  and hence the PS $\{P_n\}_{n\geqslant0}$ belongs to the class of the so-called $D_\omega$-Appell sequences.
  \smallskip
\begin{enumerate}
\item[{$\bf  A_{1a}.$}] With the choice $\hat{a}^2=2\gamma_1$ and $\hat{b}=\beta_0$,  we easily get $\hat{\beta}_n=\frac{\omega}{\hat{a}}n\ ;\
 \hat{\gamma}_{n+1}=\frac{1}{2}(n+1),\ n\geqslant0$.
  Now, replacing $\omega$ by $\hat{a}\omega$,  we obtain $\hat{\beta}_n={\omega}n\ ;\ \hat{\gamma}_{n+1}=\frac{1}{2}(n+1),\ n\geqslant0$.\\
When $\omega=0$, and so $\delta_0=0$, we meet again the Hermite polynomials.
\item[{$\bf  A_{1b}.$}] The choice $\hat{a}=\omega$ and $\hat{b}=\beta_0-\omega a$  with $\gamma_1=a\omega^2$  gives $\hat{\beta}_n=a+n\ ;\
\hat{\gamma}_{n+1}=a(n+1),\ n\geqslant0$.\\
 We thus encounter the Charlier polynomials \cite{Char}.
 \end{enumerate}
 \smallskip
\noindent{$ \bf A_2:$} $ \delta_0+\omega\ne0$. We can assume without loss of generality that $\delta_0(\delta_0+\omega)=1\Leftrightarrow
\omega=\delta_0^{-1}-\delta_0$. Use of \eqref{eq2.12}-\eqref{eq2.13} then shows that
{$\; {\beta}_n=\beta_0-(\delta_0^{-1}+\delta_0)n\ ; \ {\gamma}_{n+1}=(n+1)\big(n+\gamma_1\big),\  n\geqslant0,$}
where the parameters $\beta_0$ and $\gamma_1$ are related via $\beta_0=-\delta_0^{-1}\gamma_1$ as   in the statement {\bf I-(ii)} above. \\
 By setting  $\gamma_1:=\alpha+1$, we can write $\beta_0=-\delta_0^{-1}(\alpha+1)$. \\
 If we take now $\omega=0$ which yields $\delta_0^2=1$, we recover the Laguerre polynomials for $\delta_0=-1$, and the shifted
 Laguerre polynomials for $\delta_0=1$.\\
For $\delta_0=-1$, we  have $\ {\beta}_n=2n+\alpha+1\ ; \ {\gamma}_{n+1}=(n+1)\big(n+\alpha+1\big),\  n\geqslant0.$
 \par
\medskip
\noindent From now on, we assume that $\delta_0\ne-1$ and set $\gamma_1:=\alpha+1$ again. We  wish to  examine  two interesting situations already investigated in \cite{AbMa} with  slight differences in  notation.\par
\smallskip
 \noindent The first one occurs for $\delta_0:=-e^{-\varphi}, \varphi\ne0$, so that $\omega=-2\sinh\varphi$. It follows that
 \[
 {\beta}_n=e^{\varphi}(\alpha+1)+2n \cosh{\varphi}\ ; \  {\gamma}_{n+1}=(n+1)(n+\alpha+1) ,\ n\geqslant0.
 \]
Afterwards,  choose $\hat{a}=\omega\ ;\ \hat{b}=0$ and put $c:=e^{2\varphi}$, to get
 \[
 \hat{\beta}_n={\frac{c}{1-c}}(\alpha+1)+\frac{1+c}{1-c}n\ ; \ \hat{\gamma}_{n+1}=\frac{c}{(1-c)^2}(n+1)(n+\alpha+1),\ n\geqslant0.
  \]
  When the parameter $c\in\R\backslash\{0,1\}$, we obtain the Meixner polynomials of the first kind \cite{Meix}.
  The Krawtchouk polynomials are a special case of the Meixner polynomials of the first kind. \par
  \smallskip
  \noindent The second situation appears for $\delta_0:=e^{i\phi},\ 0<\phi<\pi$ by taking $2\lambda=\gamma_1:=\alpha+1$.
  After making a linear transformation via the changes $\hat{a}=i\omega\ ;\ \hat{b}=-\lambda\omega$, we obtain that\\
 \[
 \hat{\beta}_n=-(n+\lambda)\cot\phi\ ; \ \hat{\gamma}_{n+1}=\frac{1}{4}\frac{(n+1)(n+2\lambda)}{\sin^2\phi},\ n\geqslant0.
 \]
From this, we conclude that  the resulting polynomials are those of Meixner-Pollaczek \cite{Poll}. \par
\smallskip
\noindent {$\bf A_3:$} For $\delta_0=0$, one sees immediately that $\beta_n=\beta_0-\omega n\ ;\ \gamma_{n+1}=\gamma_1(n+1),\ n\geqslant0$.
From this, taking into consideration \eqref{eq2.1}-\eqref{eq2.2}, we check at once that $\tilde\beta_n=\beta_n-\omega $ and
$\tilde\gamma_{n+1}=\gamma_{n+1},\ n\geqslant0$.\\
In accordance with \eqref{eq1.18}, it is easily seen  that $Q_n(x)=\tau_{-\omega}P_n(x)=P_n(x+\omega)$ for all $ n\geqslant0$.\\
It is worth pointing out the following two subcases for which we can  proceed analogously to the generation of {$\bf A_{1a}$} and {$\bf A_{1b}$}.
\smallskip
\begin{enumerate}
 \item[{$\bf A_{3a.}$}] Similarly to {$\bf A_{1a}$}, the choice $\hat{a}^2=2\gamma_1\ ;\ \hat{b}=\beta_0$ yields
 $\hat{\beta}_n=-\frac{\omega}{\hat{a}}n\ ;\
  \hat{\gamma}_{n+1}=\frac{1}{2}(n+1),\  n\geqslant0$.  Replacing $\omega$ by $-\hat{a}\omega$, we get $\hat{\beta}_n={\omega}n\ ;\
  \hat{\gamma}_{n+1}=\frac{1}{2}(n+1), \ n\geqslant0$, which for $\omega=0$ gives rise to the Hermite polynomials.
\item[{$\bf  A_{3b}.$}] If we take $\hat{a}=-\omega\ ;\ \hat{b}=\beta_0+\omega a$ with $\gamma_1=a\omega^2$, we get the Charlier polynomials again.\\
Note that, with $\hat{a}=i\omega\ ;\ \hat{b}=\beta_0+i\omega b$ and $\gamma_1=-a\omega^2$, a specific case relative to this situation have been mentioned
in \cite{AbMa}, where $\hat{\beta}_n=b+in\ ;\ \hat{\gamma}_{n+1}=-a(n+1),\  n\geqslant0$.
 \end{enumerate}\par
 \noindent Note that the resulting polynomials encountered here are those obtained in {$\bf A_{1a}$} and {$\bf A_{1b}$}.\par
 \smallskip
 \noindent $\triangleright$ {\bf Case B.} Two main situations will be also  investigated with some of their special subcases.
The remarks referred to in the statements {\bf II-(iii)} and {\bf II-(iv)} above must of course be taken into account to choose more precise certain parameters involved in the recurrence coefficients.\\
 In what follows, unless otherwise stated,  we assume  that $\theta\ne1$.\\
\smallskip
 \noindent{$ \bf B_1$:} $\theta=2\alpha-1$. From \eqref{2.27} and \eqref{eq2.30}, if we choose $\beta_0=\frac{1}{2}\alpha\omega-\mu_1$,  we get
 \begin{align}
  \beta_{n}&=\frac{1}{2}\alpha\omega-\frac{\alpha(\alpha-1)\mu_1}{(n+\alpha)(n+\alpha-1)},\ n\geqslant 0,\label{eq3.1}\\
\gamma_{n+1}&=-\frac{(n+1)(n+2\alpha-1)\big(\omega (n+\alpha)^2-2\alpha\mu_1\big)^2}{(2n+2\alpha+1)(2n+2\alpha)^2(2n+2\alpha-1)},\
n\geqslant0, \label{eq3.2}
\end{align}
where we have set $\mu_1:=-\frac{1}{2}(\alpha+1)(2\delta_0+\omega)$.\\
Note that \eqref{eq3.1} is valid for $n= 0$, except that it becomes worthless if concurrently $\alpha=1$.\\
\smallskip
\noindent For $\alpha=\frac{1}{2}$, and so $\theta=0$, the coefficients $\beta_{n}$ and $\gamma_{n+1}$ coincide with those previously established in Subsection 2.1.\par
\smallskip
\noindent{$\bf  B_{11}.$} For $\mu_1\ne0$, the choice $\hat{a}=\alpha\mu_1=1\ ;\ \hat{b}=0$, leads to
\begin{align}
  \hat{\beta}_{n}&=\frac{1}{2}\alpha\omega+\frac{1-\alpha}{(n+\alpha)(n+\alpha-1)},\ n\geqslant 0,\label{eq3.3}\\
\hat{\gamma}_{n+1}&=-\frac{(n+1)(n+2\alpha-1)\big(\frac{1}{2} \omega (n+\alpha)^2-1\big)^2}{(2n+2\alpha+1)(n+\alpha)^2(2n+2\alpha-1)},\
 n\geqslant0\label{eq3.4}.
\end{align}
When $\omega=0$, we clearly encounter the Bessel polynomials.\par
\smallskip
 \noindent{$\bf B_{12}.$} For $\omega\ne0$ with the choice  $\hat{a}=i\omega\ ; \ \hat{b}=\frac{1}{2}\alpha\omega$, another specific case was discovered
 in \cite{AbMa}. If on top of that we take $\mu_1=0$ and $\alpha>0$, the corresponding orthogonal polynomials are symmetric and their associated form
 is positive definite.\par
 \smallskip
\noindent{$\bf B_2:$}  $ \theta=\nu-1$. Let  $c\in\C-\{-1\}$, thus from \label{eq2.27} and \label{eq2.31}, we obtain
 \begin{align}
  \beta_{n}&=\frac{1}{4}\nu\omega-\frac{1}{2}(1-c) +\frac{\nu(\nu-2)\mu_2}{(2n+\nu)(2n+\nu-2)},\ n\geqslant 0,\label{eq3.5}\\
\gamma_{n+1}&=-\frac{(n+1)(n+\nu-1)}{(2n+2\nu+1)(2n+2\nu)^2(2n+2\nu-1)}\times\notag\\
&\big[\omega n(n+\nu)+(1+c)n+\nu(\beta_0+1)\big]\big[\omega n(n+\nu)-(1+c)n+\nu(\beta_0-c)\big],\ n\geqslant0, \label{eq3.6}
\end{align}
where we have put $\mu_2:=\frac{1}{4}(\nu+2)(2\delta_0+\omega)$ and chosen $\beta_0:=\frac{1}{4}\nu\omega-\frac{1}{2}(1-c)+\mu_2$.\par
\medskip
\noindent{$\bf  B_{21}.$} For $\omega=0$ and  $c=1$, we may write $\nu(1+\beta_0)=2(\alpha+1)$ and $\nu(1-\beta_0)=2(\beta+1)$, so that
\[
 \nu=\alpha+\beta+2\ \ \mbox{and}\ \ \beta_0=\frac{\alpha-\beta}{\alpha+\beta+2}.
 \]
Hence, we rediscover the Jacobi polynomials whose coefficients are
 \begin{align*}
 \beta_{n}&=\frac{\alpha^2-\beta^2}{(2n+\alpha+\beta+2)(2n+\alpha+\beta)},\ n\geqslant 0,\\
\gamma_{n+1}&=4\frac{(n+1)(n+\alpha+\beta+1)(n+\alpha+1)(n+\beta+1)}{(2n+\alpha+\beta+3)(2n+\alpha+\beta+2)^2(2n+\alpha+\beta+1)},\ n\geqslant0.
\end{align*}
For $\alpha=\beta=\lambda-\frac{1}{2}$, we meet again the Gegenbauer polynomials.\par
\medskip
\noindent{$\bf B_{22}.$} When $\omega\ne0$, we first write the two expressions between square brackets in \eqref{eq3.6} as follows
\begin{align*}
\omega n(n+\nu)+(1+c)n+\nu(\beta_0+1)&=\big[\omega(n+\beta+1)+(1+c)\big](n+\alpha+1),\\
\omega n(n+\nu)-(1+c)n+\nu(\beta_0-c)&=\big[\omega(n+\alpha+1)-(1+c)\big](n+\beta+1).
\end{align*}
If we set $\eta:=\alpha-{(1+c)}/{\omega}$ and apply a linear transformation with $\hat{a}=i\omega\; ;\; \hat{b}=\frac{1}{4}\nu\omega-\frac{1}{2}(1-c)$,
 we may rewrite \eqref{eq3.5}-\eqref{eq3.6} as
\begin{align}
  \hat{\beta}_{n}&=\frac{1}{2}i(\alpha^2-\beta^2)\frac{\eta-\frac{1}{2}(\alpha+\beta)}{(2n+\alpha+\beta+2)(2n+\alpha+\beta)},\label{eq3.7}\\
  \hat{\gamma}_{n+1}&=\frac{(n+1)(n+\alpha+\beta+1)(n+\alpha+\beta+1-\eta)(n+\eta+1)(n+\alpha+1)(n+\beta+1)}
                      {(2n+\alpha+\beta+3)(2n+\alpha+\beta+2)^2(2n+\alpha+\beta+1)}.\label{eq3.8}
\end{align}
Observe that if $\alpha^2-\beta^2=0$ or $\eta=\frac{1}{2}(\alpha+\beta)$, the corresponding polynomials are symmetric. In consequence, it is worth while
 to discuss the following subcases mentioned in \cite{AbMa}.\par
\medskip
\noindent{$\bf  B_{22a}$.} For $\alpha-\beta=0$, we obtain
\begin{align}
\hat{\gamma}_{n+1}=\frac{1}{4}\frac{(n+1)(n+2\alpha+1)(n+2\alpha+1-\eta)(n+\eta+1)}{(2n+2\alpha+3)(2n+2\alpha+1)},\ n\geqslant0.\label{eq3.9}
\end{align}
 When $-1<\eta<2\alpha+1$ or when $\alpha\in \R$ and $\eta+\bar{\eta}=2\alpha,\ \alpha+1>0$, the obtained polynomials are orthogonal with respect to a positive definite form.\\
- For $\alpha=0$, the resulting polynomials are related to Pasternak polynomials \cite{Past} having the coefficients
\begin{align}
\hat{\gamma}_{n+1}=\frac{1}{4}\frac{(n+1)^2(n+\eta+1)(n-\eta+1)}{(2n+3)(2n+1)},\ n\geqslant0.\label{eq3.10}
\end{align}.\\
- For $\eta=0$, the resulting polynomials are the Touchard ones \cite{Touc} which are themselves particular cases of the continuous Hahn polynomials
\cite{Aske}. In this case we have
\begin{align}
\hat{\gamma}_{n+1}=\frac{1}{4}\frac{(n+1)^2(n+\alpha+1)(n-\alpha+1)}{(2n+3)(2n+1)},\ n\geqslant0.\label{eq3.11}
\end{align}\par
\medskip
\noindent {$\bf  B_{22b}$.} For $\alpha+\beta=0$, we get
\begin{align}
\hat{\gamma}_{n+1}=\frac{1}{4}\frac{(n+\alpha+1)(n-\alpha+1)(n+\eta+1)(n-\eta+1)}{(2n+3)(2n+1)},\ n\geqslant0.\label{eq3.12}
\end{align}
Again, observe that the associated form to these orthogonal polynomials is positive definite when $-1<\eta<2\alpha+1$ or when $\alpha\in \R$ and
$\eta+\bar{\eta}=2\alpha,\ \alpha+1>0$.\\
- When $\alpha=0$, this leads to the Pasternak polynomials.\\
- When $\eta=0$, we meet again the Touchard polynomials. \par
\medskip
\noindent{$\bf  B_{22c}$.} For $\eta=\frac{1}{2}(\alpha+\beta)$, we get
\begin{align}
\hat{\gamma}_{n+1}=\frac{1}{4}\frac{(n+1)(n+\alpha+\beta+1)(n+\alpha+1)(n+\beta+1)}{(2n+\alpha+\beta+3)(2n+\alpha+\beta+1)},\ n\geqslant0.\label{eq3.13}
\end{align}
For $\alpha+1>0$ and $\beta+1>0$, we have $\hat{\gamma}_{n+1}>0$, for all $n\geqslant0$. The form associated to these polynomials is then positive definite.\\
 On the other hand, if both $\alpha$ and $\beta$ are assumed to be real numbers, and so $\eta+\bar{\eta}=\alpha+\beta$, it is possible to find two 
 numbers $a$ and $b$, such that\\
$$
a+\bar{a}=\alpha+1,\, b+\bar{b}=\beta+1,\, a+\bar{b}=\eta+1 \ \mbox{and}\ \bar{a}+b=\alpha+\beta+1-\eta.
$$
In addition, under the   conditions  $\Re{a}>0$ and $\Re{b}>0$, it was shown in \cite[p.18]{AbMa} that the resulting polynomials are orthogonal w.r.t. 
a positive definite form and coincide with the continuous Hahn polynomials. For more details we refer the reader to the aforementioned paper.
   \section{The recurrence coefficients of the higher  order  derivatives}
Let $k$ be a positive integer and let  $\{P_n\}_{n\geqslant0}$ be a $D_\omega$-classical OPS.
The sequence of the normalized higher order $D_\omega$-derivatives, denoted as $\{P^{[k]}_n\}_{n\geqslant0}$, is recursively defined by
 \begin{subequations}
  \begin{align}
 &P^{[k]}_n(x):=\frac{1}{n+1} D_\omega P^{[k-1]}_{n+1}(x), \ \ k\geqslant1,\label{eq4.1a}\\
 \text{or, equivalently,}\hskip2cm \notag\\
 &P^{[k]}_n(x):=\frac{1}{(n+1)_k} D^k_\omega P_{n+k}(x), \ \ k\geqslant1,\hskip3cm \label{eq4.1b}
  \end{align}
  \end{subequations}
where $(\mu)_n=\mu(\mu+1)\cdots(\mu+n-1), \, (\mu)_0=1, \ \mu\in\C, \, n\in\N$,  is the Pochhammer symbol.
Application of Definition 1.1,  with the special notations $P^{[1]}_n:=Q_n$ and $P^{[0]}_n:=P_n$, enables one  to write $\beta^{[1]}_n:=\tilde\beta_n$, $\gamma^{[1]}_{n}:=\tilde\gamma_{n}$ and  $\beta^{[0]}_n:= \beta_n$ , $\gamma^{[0]}_{n}:=\gamma_{n}$, respectively.\\
The following corollary, which is in fact an immediate consequence of Proposition 1.2, plays an important role in establishing our results.
  \begin{corollary}[\cite{AbMa}]
If the OPS $\{P_n\}_{n\geqslant0}$ is $D_\omega$-classical, then the sequence $\{P^{[k]}_n\}_{n\geqslant0}$ is also $D_\omega$-classical OPS for any
  $k\geqslant1$.
  \end{corollary}
By an application of this corollary, if we denote by $\big(\beta^{[k]}_n , \gamma^{[k]}_{n+1}\big)_{n\in\N}$ the recurrence coefficients corresponding 
to the OPS $\{P^{[k]}_n\}_{n\geqslant0}$, with $k\geqslant1$, then
 \begin{subequations}
\begin{align}
&P^{[k]}_{n+2}(x) = (x - \beta^{[k]}_{n+1})P^{[k]}_{n+1}(x)-\gamma^{[k]}_{n+1}P^{[k]}_n (x), \ n\geqslant0,\label{4.2a}\\
&P^{[k]}_1(x) = x - \beta^{[k]}_0,\ P^{[k]}_0(x) = 1.\label{4.2b}
\end{align}
\end{subequations}
 Our objective here is to express the coefficients $\beta^{[k]}_n$ and $\gamma^{[k]}_{n+1}$  in terms of the corresponding coefficients of the OPS
 $\{P_n\}_{n\geqslant0}$, namely, $\beta_n$ and  $\gamma_{n+1}$ obtained either in Case {\bf A} or in Case {\bf B}.\\
This will be stated in the next proposition.
\begin{proposition} Let $\{P_n\}_{n\geqslant0}$ be a $D_\omega$-classical OPS. Then, for every  $k\in \N$, we have
  \begin{align}
\beta^{[k]}_{n}&=\beta_{n+k}+k\delta_{n+k-1},\ \ n\geqslant0,\label{4.3}\\
\gamma^{[k]}_n &= \frac{n}{n+k}\gamma_{n+k}\big(k(\theta_{n+k-1}-1)+1\big), \ \ n\geqslant1,\label{4.4}
 \end{align}
  where  $\delta_n$ and $\theta_n$ are solutions for  the equations \eqref{eq2.7} and \eqref{eq2.8}, respectively.
   \end{proposition}
\begin{proof} From Corollary 4.1 it follows that each sequence  $\{P^{[k]}_n\}_{n\geqslant0}$, $k\geqslant1$, is also $D_\omega$-classical.
Accordingly, both of the OPS $\{P^{[k]}_n\}_{n\geqslant0}$ and $\{P^{[k+1]}_n\}_{n\geqslant0}$ are characterized by the fact that they satisfy 
a second structure relation of type \eqref{1.15a}-\eqref{1.15b}. For $\{P^{[k]}_n\}_{n\geqslant0}$, this  is given by
\begin{subequations}\begin{align}
&P^{[k-1]}_{n+2} = P^{[k]}_{n+2} + \alpha^{k,1}_{n+1}P^{[k]}_{n+1} + \alpha^{k,0} _{n}P^{[k]}_{n}, \ n\geqslant0, \label{4.5a}\\
&P^{[k-1]}_1 =  P^{[k]}_1 +\alpha^{k,1}_0 ,\ \  P^{[k-1]}_0 =  P^{[k]}_0 = 1,\label{4.5b}
\end{align}\end{subequations}
where
\begin{align}
\alpha^{k,1}_{n}=(n+1)\left(\beta^{[k-1]}_{n+1}-\beta^{[k]}_{n}-\omega\right) \ \ \mbox{and}\ \ \alpha^{k,0} _{n}=(n+1)\gamma^{[k-1]}_{n+2}-
(n+2)\gamma^{[k]}_{n+1}.\label{4.6}
\end{align}
Likewise, for $\{P^{[k+1]}_n\}_{n\geqslant0}$, an equivalent  relation  may be written as
\begin{subequations}\begin{align}
&P^{[k]}_{n+2} = P^{[k+1]}_{n+2} + \alpha^{k+1,1}_{n+1}P^{[k+1]}_{n+1} + \alpha^{k+1,0} _{n}P^{[k+1]}_{n}, \ n\geqslant0, \label{4.7a}\\
&P^{[k]}_1 =  P^{[k+1]}_1 +\alpha^{k+1,1}_0 ,\ \  P^{[k]}_0 =  P^{[k+1]}_0 = 1,\label{4.7b}
\end{align}\end{subequations}
with
\begin{align}
\alpha^{k+1,1}_{n}=(n+1)\left(\beta^{[k]}_{n+1}-\beta^{[k+1]}_{n}-\omega\right) \ \ \mbox{and}\ \ \alpha^{k+1,0} _{n}=(n+1)\gamma^{[k]}_{n+2}-
(n+2)\gamma^{[k+1]}_{n+1}.\label{4.8}
\end{align}
To prove the equalities \eqref{4.3} and \eqref{4.4}, we can proceed by induction on $k$. For this, let $\mathrm{P}(k)$ be the proposition
that these two equalities  hold. Observe first that the assertion $\mathrm{P}(1)$ is  trivial.
 Let us check that  $\mathrm{P}(2)$ is true. Setting $k=1$ in \eqref{4.5a}-\eqref{4.5b} we get
 \begin{subequations}\begin{align}
&P_{n+2} = P^{[1]}_{n+2} + \alpha^{1,1}_{n+1}P^{[1]}_{n+1} + \alpha^{1,0} _{n}P^{[1]}_{n}, \ n\geqslant0, \label{4.9a}\\
&P_1 =  P^{[1]}_1 +\alpha^{1,1}_0 ,\ \  P_0 =  P^{[1]}_0 = 1,\label{4.9b}
\end{align}\end{subequations}
where $\alpha^{1,1}_{n}:=\tilde{\alpha}^{1}_{n}$ and $\alpha^{1,0}_{n}:=\tilde{\alpha}^{0}_{n}$, since $P^{[1]}_n:=Q_n$ and $P^{[0]}_n:=P_n$.
Roughly  speaking, in this  special case,  the formulas \eqref{4.9a}-\eqref{4.9b} and \eqref{1.15a}-\eqref{1.15b} coincide.  \\
Similarly,  if we take $k=2$ in \eqref{4.5a}-\eqref{4.5b}, we have
 \begin{subequations}\begin{align}
&P^{[1]}_{n+2} = P^{[2]}_{n+2} + \alpha^{2,1}_{n+1}P^{[2]}_{n+1} + \alpha^{2,0} _{n}P^{[2]}_{n}, \ n\geqslant0, \label{4.10a}\\
&P^{[1]}_1 =  P^{[2]}_1 +\alpha^{2,1}_0 ,\ \  P^{[1]}_0 =  P^{[2]}_0 = 1.\label{4.10b}
\end{align}\end{subequations}
 Replace $n$ by $n+1$ in \eqref{4.9a} and then apply the operator $D_\omega$ yields
\begin{align}
(n+3)P^{[1]}_{n+2} = (n+3)P^{[2]}_{n+2} + (n+2)\alpha^{1,1}_{n+2}P^{[2]}_{n+1} + (n+1)\alpha^{1,0} _{n+1}P^{[2]}_{n}, \ n\geqslant0. \label{4.11}
\end{align}
Multiply both sides of \eqref{4.10a} by $(n+3)$ and compare this with \eqref{4.11} readily gives
\begin{subequations}\begin{align}
(n+3)\alpha^{2,1}_{n+1}&=(n+2)\alpha^{1,1}_{n+2},\label{4.12a}\\
(n+3)\alpha^{2,0}_{n}&=(n+1)\alpha^{1,0}_{n+1}. \label{4.12b}
\end{align}\end{subequations}
By  \eqref{4.6}, for $k=2$  and $k=1$, it is easy to check that  \eqref{4.12a} and \eqref{4.12b}, respectively,  lead to
\begin{align*}
(n+3)(n+2)\big[\beta^{[1]}_{n+2}-\beta^{[2]}_{n+1}-\omega\big]&=(n+3)(n+2)\big[\beta_{n+3}-\beta^{[1]}_{n+2}-\omega\big], \\
(n+3)\big[(n+1)\gamma^{[1]}_{n+2}- (n+2)\gamma^{[2]}_{n+1}\big]&=(n+1)\big[(n+2)\gamma_{n+3}- (n+3)\gamma^{[1]}_{n+2}\big].
\end{align*}
 Thanks to  \eqref{eq2.1}-\eqref{eq2.2}, we respectively deduce that
\begin{align*}
\beta^{[2]}_{n}&=\beta_{n+2}+2\delta_{n+1},\ \ n\geqslant0,\\
\gamma^{[2]}_n &= \frac{n}{n+2}\gamma_{n+2}\big(2\theta_{n+1}-1\big), \ \ n\geqslant1,
 \end{align*}
 which is precisely the desired conclusion, that is,  $\mathrm{P}(2)$ is true.\par
 \noindent
Now, we must show that the conditional statement  $\mathrm{P}(k)\rightarrow\mathrm{P}(k+1)$ is true for all positive integers $k$.
We can proceed analogously to the proof of $\mathrm{P}(2)$.
Starting from \eqref{4.5a}, replacing $n$ by $n+1$ and then apply the operator $D_\omega$ we find
\begin{align}
(n+3)P^{[k]}_{n+2} = (n+3)P^{[k+1]}_{n+2} + (n+2)\alpha^{k,1}_{n+2}P^{[k+1]}_{n+1} + (n+1)\alpha^{k,0} _{n+1}P^{[k+1]}_{n}, \ n\geqslant0. \label{4.13}
\end{align}
Multiply both sides of \eqref{4.7a} by $(n+3)$ and compare this with \eqref{4.13} readily gives
\begin{subequations}\begin{align}
(n+3)\alpha^{k+1,1}_{n+1}&=(n+2)\alpha^{k,1}_{n+2},\label{4.14a}\\
(n+3)\alpha^{k+1,0}_{n}&=(n+1)\alpha^{k,0}_{n+1}. \label{4.14b}
\end{align}\end{subequations}
On account of \eqref{4.6}-\eqref{4.8}, we have
\begin{align*}
(n+3)(n+2)\big[\beta^{[k]}_{n+2}-\beta^{[k+1]}_{n+1}-\omega\big]&=(n+3)(n+2)\big[\beta^{[k-1]}_{n+3}-\beta^{[k]}_{n+2}-\omega\big], \\
(n+3)\big[(n+1)\gamma^{[k]}_{n+2}- (n+2)\gamma^{[k+1]}_{n+1}\big]&=(n+1)\big[(n+2)\gamma^{[k-1]}_{n+3}- (n+3)\gamma^{[k]}_{n+2}\big].
\end{align*}
From this, we deduce that
\begin{align}
 \beta^{[k+1]}_{n+1} &=2\beta^{[k]}_{n+2}-\beta^{[k-1]}_{n+3},\label{4.15} \\
\gamma^{[k+1]}_{n+1} &=2\frac{n+1}{n+2}\gamma^{[k]}_{n+2}-\frac{n+1}{n+3}\gamma^{[k-1]}_{n+3}.\label{4.16}
\end{align}
On the other hand, according to the induction hypothesis  we may write
    \begin{align*}
\beta^{[k]}_{n+2}&=\beta_{n+k+2}+k\delta_{n+k+1},
&&\beta^{[k-1]}_{n+3}=\beta_{n+k+2}+(k-1)\delta_{n+k+1},\\
\gamma^{[k]}_{n+2} \!&=  \!\frac{n\!+ \!2}{n\!+ \!k\!+ \!2}\gamma_{n+k+2}\big(k(\theta_{n+k+1} \!- \!1) \!+ \!1\!\big),\
&&\gamma^{[k-1]}_{n+3} \! = \! \frac{n\!+ \!3}{n\!+ \!k\!+ \!2}\gamma_{n+k+2}\big(\!(k\!- \!1)(\theta_{n+k+1}\!- \!1)\!+ \!1\!\big).
  \end{align*}
 Substituting these into  \eqref{4.15}-\eqref{4.16}, and then changing $n$ into $n-1$, it follows  that

   \begin{align*}
\beta^{[k+1]}_{n}&=\beta_{n+k+1}+(k+1)\delta_{n+k},\ \ n\geqslant0,\\
\gamma^{[k+1]}_n &= \frac{n}{n+k+1}\gamma_{n+k+1}\big((k+1)(\theta_{n+k}-1)+1\big), \ \ n\geqslant1.
 \end{align*}
 These last equalities show that $\mathrm{P}(k+1)$ is also true, which completes the proof.
\end{proof}\par
\medskip
\noindent{\bf Remark.}
Due to \eqref{4.3} and \eqref{4.4}, from \eqref{4.6}, it is seen that the two coefficients $\alpha^{k,1}_{n}$ and $\alpha^{k,0}_{n}$ involving 
in the structure relation \eqref{4.5a}-\eqref{4.5b} can be rewritten as
 \begin{align}
\alpha^{k,1}_{n}=-(n+1)\left(\delta_{n+k-1}+\omega\right) \, ;\, \alpha^{k,0} _{n}=\frac{(n+1)(n+2)}{n+k+1}\gamma_{n+k+1}
\left(1-\theta_{n+k}\right),\ n\geqslant0. \label{eq4.17}
\end{align}
For $k=1$, this reduces to  \eqref{eq2.3} providing the coefficients of \eqref{1.15a}-\eqref{1.15b}. \par
\medskip
\noindent{\bf Application.}\par
\smallskip
\noindent When  $\omega=0$, the identities \eqref{4.5a}-\eqref{4.5b} consist of the structure relation characterising the higher order derivatives sequence  of the classical orthogonal polynomials. In this case, a direct application of the preceding  proposition enables us to express  the recurrence coefficients of the sequence $\{P^{[k]}_n\}_{n\geqslant0}$ in terms of the recurrence coefficients for each of the four  classical families.\\
To this purpose, if we denote by $\{\hat{H}_n\}_{n\geqslant0}$, $\{\hat{L}_n(.;\alpha)\}_{n\geqslant0}$, $\{\hat{B}_n(.;\alpha)\}_{n\geqslant0}$ 
and $\{\hat{J}_n(.;\alpha,\beta)\}_{n\geqslant0}$, respectively, the (monic) Hermite, Laguerre, Bessel and Jacobi polynomials, and by $\hat{H}^{[k]}_n(x)$,  $\hat{L}^{[k]}_n(x;\alpha)$, $\hat{B}^{[k]}_n(x;\alpha)$ and $\hat{J}^{[k]}_n(x;\alpha,\beta)$ their corresponding sequences of derivatives of order
  $k$, then application of Formulas \eqref{4.3} and \eqref{4.4} successively give:\par
 \medskip
\noindent   {\bf Case A.}  For $\ \theta_n=1,\, n\geqslant1$, and $\delta_n=\delta_0,\, n\geqslant1$,  we have 
    \begin{align}
\hat{\beta}^{[k]}_{n}&=\hat{\beta}_{n+k}+k\delta_{0},\ \ n\geqslant0,\label{eq4.18}\\
\hat{\gamma}^{[k]}_n &= \frac{n}{n+k}\hat{\gamma}_{n+k}, \ \ n\geqslant1.\label{eq4.19}
 \end{align}
{\it Hermite case} : When $\delta_0=0$, $\hat{\gamma}_1=\frac{1}{2}$, this yields
\begin{align*}
   \hat{\beta}^{[k]}_n=0, n\geqslant0, \ \ \mbox{and}\ \ \hat{\gamma}^{[k]}_{n+1}=\frac{1}{2}(n+1), n\geqslant0.
 \end{align*}
 \smallskip
 \noindent{\it Laguerre case} :  When $\delta_0=-1$, $\hat{\gamma}_1=\hat{\beta}_0=\alpha+1$, we get
 \begin{align*}
 \hat{\beta}^{[k]}_n=2n+\alpha+k+1, n\geqslant0, \ \ \mbox{and}\ \ \hat{\gamma}^{[k]}_{n+1}=(n+1)(n+\alpha+k+1), n\geqslant0.
  \end{align*}
    \medskip
\noindent  {\bf Case B.} For  $ \theta_n=\frac{n+\theta+1}{n+\theta},\, n\geqslant1, \  \mbox{and}\
  \delta_n=\frac{\delta_0(\theta+3)(\theta+1)}{\big(2n+\theta+3\big)\big(2n+\theta+1\big)},\, n\geqslant0$, we deduce that
  \begin{align}
\hat{\beta}^{[k]}_{n}&=\hat{\beta}_{n+k}+\frac{k\delta_0(\theta+3)(\theta+1) }{\big(2(n+k)+\theta+1\big)\big(2(n+k)+\theta-1\big)} ,\ \ n\geqslant0,\label{eq4.19}\\
\hat{\gamma}^{[k]}_n &= \frac{n(n+\theta+2k-1)}{(n+k)(n+\theta+k-1)}\hat{\gamma}_{n+k} , \ \ n\geqslant1.\label{eq4.30}
 \end{align}
 \noindent{\it Bessel case} : When $\theta=2\alpha-1$, $\delta_0=-1/(\alpha+1)\alpha$, we obtain
  \begin{align*} \hat{\beta}^{[k]}_{n}&=\frac{1-(\alpha+k)}{(n+\alpha+k)(n+\alpha+k-1)},\  n\geqslant0, \\
\hat{\gamma}^{[k]}_{n+1}&=-\frac{(n+1)(n+2(\alpha+k)-1)}{(2n+2(\alpha+k)+1)(n+\alpha+k)^2(2n+2(\alpha+k)-1)},\ n\geqslant0.
 \end{align*}
    \noindent{\it Jacobi case} : When $\theta=\alpha+\beta+1$, $\delta_0=2(\alpha-\beta)/(\alpha+\beta+4)(\alpha+\beta+2)$, we get
  \begin{align*}
  \hat{\beta}^{[k]}_{n}&=\frac{(\alpha-\beta)(\alpha+\beta+2k)}{(2n+\alpha+\beta+2k+2)(2n+\alpha+\beta+2k)},\ n\geqslant 0,\\
\hat{\gamma}^{[k]}_{n+1}&=4\frac{(n+1)(n+\alpha+\beta+2k+1)(n+\alpha+k+1)(n+\beta+k+1)}
{(2n+\alpha+\beta+2k+3)(2n+\alpha+\beta+2k+2)^2(2n+\alpha+\beta+2k+1)},\ n\geqslant0.
\end{align*}
We thus rediscover  the well known relations $ \hat{H}^{[k]}_n(x)=\hat{H}_n(x)$, $\hat{L}^{[k]}_n(x;\alpha)=\hat{L}_n(x;\alpha+k)$,
 $\hat{B}^{[k]}_n(x;\alpha)=\hat{B}_n(x;\alpha+k)$ and  $\hat{J}^{[k]}_n(x;\alpha,\beta)=\hat{J}_n(x;\alpha+k,\beta+k)$.\\
 The results presented above are of course known and clearly show that the sequences of higher order derivative for the classical
 orthogonal polynomials belonging to the same class, provided that the parameters $\alpha$ and $\beta$ take values in the range of regularity.\par
 \bigskip
\noindent {\bf Conclusion.}\\
 We studied the $D_\omega$-classical orthogonal polynomials using a new method in this domain. The results obtained in Section 3 are  expected and 
 are consistent with those found in \cite{AbMa}, where four different families are pointed out with some of their special cases. The recurrence coefficients of the resulting orthogonal polynomials are explicitly determined. Proposition 1.4 established a new characterization of these polynomials via a structure relation, and   Proposition 4.2  provided  relations  connecting the  recurrence coefficients of each sequence of polynomials with those of its higher-order derivatives. For $\omega=0$, the classical orthogonal polynomials are rediscovered.

\end{document}